\documentclass{jams-l-m}

\usepackage{amssymb}

\usepackage{dsfont}

\usepackage{mathabx}

\usepackage{graphicx}

\usepackage{hyperref}

\usepackage{xcolor}

\usepackage{mathbbol}

\usepackage{array} 

\usepackage{float}

\usepackage{soul} 

\usepackage[autostyle]{csquotes} 

\newtheorem{theorem}{Theorem}[section]

\theoremstyle{definition}

\theoremstyle{remark}

\numberwithin{equation}{section}

\DeclareSymbolFontAlphabet{\amsmathbb}{AMSb}%

\usepackage{tikz}

\usetikzlibrary{arrows,positioning,calc} 
\usetikzlibrary{decorations.pathreplacing}
\tikzset{
    >=stealth',
    punkt/.style={
           rectangle,
           rounded corners,
           draw=black, thick,
           text width=5.5em,
           minimum height=2em,
           text centered},
    punktl/.style={
           rectangle,
           rounded corners,
           draw=black, thick,
           text width=7em,
           minimum height=2em,
           text centered},
    pil/.style={
           ->,
           shorten <=4pt,
       shorten >=4pt
    },
    pildotted/.style={
           ->,
           shorten <=4pt,
           shorten >=4pt,
  dotted,
  },
    punktf/.style={
           rectangle,
           text width=4.0em,
           minimum height=1.5em,
           text centered},
    punktfTop/.style={
           rectangle,
           text width=4.0em,
           minimum height=1.5em,
           text centered,
           append after command={
               [thick,shorten >=0.2bp, shorten <=0.2bp]
               (\tikzlastnode.north west)edge(\tikzlastnode.north east)
}
    },
    punktfBot/.style={
           rectangle,
           text width=4.0em,
           minimum height=1.5em,
           text centered,
           append after command={
               [thick,shorten >=0.2bp, shorten <=0.2bp]
               (\tikzlastnode.south west)edge(\tikzlastnode.south east)
            }
    }
}

\usepackage{prodint}
\theoremstyle{plain}

\makeatletter
\newcommand*\bigcdot{\mathpalette\bigcdot@{.65}}
\newcommand*\bigcdot@[2]{\mathbin{\vcenter{\hbox{\scalebox{#2}{$\m@th#1\bullet$}}}}}
\makeatother

\usepackage[shortlabels]{enumitem} 

\usepackage{changepage}

\begin{document}

\title[Existence and uniqueness for non-linear forward equations]{Existence and uniqueness for non-linear forward equations: Addendum to ``Mean-field approximations in insurance'' (Hornung, 2026)}


\author{Christian Furrer}
\address{Department of Mathematical Sciences, University of Copenhagen, Universitetsparken 5, DK-2100 Copenhagen, Denmark}
\curraddr{}
\email{\href{mailto:furrer@math.ku.dk}{furrer@math.ku.dk}}
\thanks{}

\author{Philipp C.\ Hornung}
\address{Department of Mathematical Sciences, University of Copenhagen, Universitetsparken 5, DK-2100 Copenhagen, Denmark}
\curraddr{}
\email{\href{mailto:pcho@math.ku.dk}{pcho@math.ku.dk}}
\thanks{}


\date{}

\dedicatory{\today}

\begin{abstract}
Existence and uniqueness of solutions to non-linear forward integro-differential equations related to distribution-dependent Markov jump processes is established under standard boundedness and Lipschitz conditions. \\


\end{abstract}

\maketitle



\section{Introduction}

Non-linear forward integro-differential equations are widely used to characterize the forward dynamics of cohorts when a mean-field model is adopted to capture interactions among individuals. In actuarial science, key areas of application of mean-field approximations include epidemic health insurance, see~\cite{FengGarrdio2011,FengEtAl2022,FrancisSteffensen2024}, as well as cyber insurance, confer with~\cite{FahrenwaldtEtAl2018,HillairetLopez2021,HillairetEtAl2022}.

Recently, in~\cite{Hornung2026}, mean-field approximations for a wide class of pure jump processes were studied in detail. In particular, it was demonstrated that, under rather mild regularity conditions, the limiting process becomes a distribution-dependent Markov jump process with occupation probabilities governed by a specific non-linear forward integro-differential equation. However, the corresponding question of uniqueness for this non-linear forward integro-differential equation was left open, confer with Remark~3.12 in~\cite{Hornung2026}. This is of course unsatisfactory, not least from an applied (numerical) perspective. 

The present paper serves as an addendum to~\cite{Hornung2026} and completely resolves this unsettled question in the affirmative. We establish both existence and uniqueness under standard boundedness and Lipschitz assumptions by first reformulating the non-linear forward integro-differential equation as a fixed-point problem for a solution mapping on a complete metric space of paths of probability measures and then showing that this mapping defines a contraction when the space is endowed with an appropriate family of time-weighted uniform total variation norms.

The mathematical framework, along with the precise formulation of the theorem and a detailed proof, is presented in the next section.

\section{Existence and uniqueness}
Consider a Borel space $(X, \mathcal{B}(X))$ and let $\mathcal{P}(X)$ denote the set of probability measures on $\mathcal{B}(X)$. In~\cite{Hornung2026}, it was shown that a wide class of Markov jump processes with values in the product space $X^n$, $n\gg 0$, each coordinate can be approximated by a distribution-dependent Markov jump process with values in $X$ whose occupation probabilities satisfy a non-linear forward integro-differential equation. For an initial distribution $\zeta\in\mathcal{P}(X)$, a terminal time $T>0$, and a distribution-dependent transition intensity kernel $\mu:[0,T]\times X\times\mathcal{P}(X)\times \mathcal{B}(X)\rightarrow [0,1]$ satisfying
\begin{enumerate}
       \item[(i)] $(t,x,\sigma)\mapsto\mu_t(x,\sigma,A)$ is measurable for all $A\in\mathcal{B}(X)$
       \item[(ii)] $A\mapsto\mu_t(x,\sigma,A)$ is a measure for all $(t,x,\sigma)\in[0,T]\times X\times\mathcal{P}(X)$
       \item[(iii)] $\mu_t(x,\sigma,\{x\})=0$ for all $x\in X$,
\end{enumerate}
the corresponding non-linear forward integro-differential equation reads
\begin{align}\label{eq:nonlinear}
\begin{split}
\frac{\partial}{\partial t} p_t(A) &= \int_X \mu_t(x,p_t,A) \, p_t(\mathrm{d}x) - \int_{A} \lambda_t(x,p_t) \, p_t(\mathrm{d}x), \quad t\in(0,T],\\
p_0(A)&=\zeta(A),
\end{split}
\end{align}
for $A \in \mathcal{B}(X)$; here $\lambda$ is the total transition intensity: $\lambda_t(x,\sigma) := \mu_t(x,\sigma,X\setminus \{x\})$. 

In the case where the intensity kernel $\mu$ is of the form $\mu_t(x,\sigma,\mathrm{d}y) \equiv \mu_t(x,\mathrm{d}y)$ -- that is, it has no distribution dependence -- the forward integro-differential equation~\eqref{eq:nonlinear} reduces to the familiar (linear) Kolmogorov forward integro-differential equation
\begin{align}\label{eq:linear}
\begin{split}
\frac{\partial}{\partial t} p_t(A) &= \int_X \mu_t(x, A) \, p_t(\mathrm{d}x) - \int_{A} \lambda_t(x) \, p_t(\mathrm{d}x), \quad t\in(0,T],\\
p_0(A)&=\zeta(A),
\end{split}
\end{align}
for $A \in \mathcal{B}(X)$. In this case, a first existence result was established in~\cite{Feller1940}, while the question of uniqueness was clarified in~\cite{FeinbergMandavaShiryaev2014,FeinbergMandavaShiryaev2022}. In particular, the following result was proven.
\begin{theorem}\label{th:linear}
       Suppose that $\sup_{t\in[0,T],x\in X}\lambda_t(x)<\infty$. Then there exists a unique solution to~\eqref{eq:linear}.
\end{theorem}

We now set out to prove an existence and uniqueness result for the non-linear version~\eqref{eq:nonlinear}. In the following, let $\mathcal{S}(X)$ denote the space of signed measures on $\mathcal{B}(X)$ and denote the total variation norm of any signed measure $\nu\in\mathcal{S}(X)$ by $\lVert \nu \rVert_{\text{TV}}$. Recall that
\begin{align*}
\lVert \nu \rVert_{\text{TV}} = \sup_{f \in B_1(X)} \bigg| \int_X f(x) \, \nu(\mathrm{d}x)\bigg|,
\end{align*}
where $B_1(X)$ is the space of Borel-measurable functions from $X$ to $[-1,1]$. The space $(\mathcal{S}(X),\|\cdot\|_{\text{TV}})$ is a Banach space. Further, let $C([0,T],\mathcal{S}(X))$ denote the set of all continuous functions $\nu:[0,T]\rightarrow\mathcal{S}(X)$, which equipped with the uniform norm 
\begin{align*}
       \|\nu\|_{\text{TV}}^{\infty}:=\sup_{t\in[0,T]}\|\nu_t\|_{\text{TV}}
\end{align*}
is a Banach space. For technical reasons, we also make use of the following family of equivalent weighted uniform norms:
\begin{align*}
       \|\nu\|_{\text{TV}}^{\infty,\beta}:=\sup_{t\in[0,T]}e^{-\beta t}\|\nu_t\|_{\text{TV}},\quad \beta>0.
\end{align*}
The set $C([0,T],\mathcal{P}(X))$ of all continuous functions $\sigma:[0,T]\rightarrow\mathcal{P}(X)$ is a closed subset of $C([0,T],\mathcal{S}(X))$ and thus a complete metric space when endowed with the metric induced by $\|\cdot\|_{\text{TV}}^{\infty}$ or $\|\cdot\|_{\text{TV}}^{\infty,\beta}$.

\begin{theorem}
Suppose that the following two conditions are satisfied:
\begin{itemize}
\item[(B1)] $\sup_{t \in [0,T], x \in X, \sigma \in \mathcal{P}(X)} \lambda_t(x,\sigma) < \infty$. 
\item[(L1)] There exists a constant $K > 0$ such that for all $t\in[0,T]$, $x \in X$, and $\sigma, \sigma' \in \mathcal{P}(X)$ it holds that
\begin{align*}
\lVert \mu_t(x,\sigma,\cdot) - \mu_{t}(x,\sigma',\cdot) \rVert_{\text{TV}} \leq K \lVert \sigma - \sigma' \rVert_{\text{TV}}.
\end{align*}
\end{itemize}
Then there exists a unique solution to~\eqref{eq:nonlinear}.
\end{theorem}
\begin{proof}
       Define the mapping $\mathcal{M}:C([0,T],\mathcal{P}(X))\rightarrow C([0,T],\mathcal{P}(X))$ with $\mathcal{M}(\sigma)=(p_t^{\sigma})_{t\in[0,T]}$ as follows:
       \begin{align}\label{eq:help1}
              p_t^{\sigma}(A)=\zeta(A)+\int_0^t\bigg[\int_X\mu_s(x,\sigma_s,A)p_s^{\sigma}(\mathrm{d}x)-\int_A\lambda_s(x,\sigma_s)p_s^{\sigma}(\mathrm{d}x)\bigg]\mathrm{d}s
       \end{align}
       for $A\in\mathcal{B}(X)$ and $t\in[0,T]$. Since~\eqref{eq:help1} is of the form~\eqref{eq:linear}, Theorem~\ref{th:linear} ensures that the mapping $\mathcal{M}$ indeed for each input $\sigma\in C([0,T],\mathcal{P}(X))$ returns a unique output $\mathcal{M}(\sigma)\in C([0,T],\mathcal{P}(X))$. The goal is to prove that there exists some $\beta>0$ and $0<C<1$ such that, for all $\sigma, \sigma' \in C([0,T],\mathcal{P}(X))$, it holds that
       \begin{align*}
              \|\mathcal{M}(\sigma)-\mathcal{M}(\sigma')\|_{\text{TV}}^{\infty,\beta}\leq C\|\sigma-\sigma'\|_{\text{TV}}^{\infty,\beta}.
       \end{align*} 
       This would show that $\mathcal{M}$ is a contraction, in which case our desired result follows directly from the Banach fixed-point theorem.
       
       Let $\sigma,\sigma'\in C([0,T],\mathcal{P}(X))$ and define $\Delta:=\mathcal{M}(\sigma)-\mathcal{M}(\sigma')\in C([0,T],\mathcal{S}(X))$ with $\Delta_t=p_t^{\sigma}-p_t^{\sigma'}\in\mathcal{S}(X)$. Note that $\Delta_0$ is the null-measure. Using this notation, we may write 
       \begin{align*}
              \|\mathcal{M}(\sigma)-\mathcal{M}(\sigma')\|_{\text{TV}}^{\infty,\beta}=\sup_{t\in[0,T]}e^{-\beta t}\sup_{\phi\in B_1(X)}\bigg|\int_X\phi(x)\Delta_t(\mathrm{d}x)\bigg|,
       \end{align*}
	A natural starting point is thus the right-hand side integral.
       
       Utilising~\eqref{eq:help1}, a standard simple function approximation argument combined with the dominated convergence theorem yields
       \begin{align*}
              \frac{\partial}{\partial t}\int_X \phi(x) \, p_t^{\sigma}(\mathrm{d}x) =& \int_X \bigg[ \int_X \phi(y) \, \mu_s(x, \sigma_s, \mathrm{d}y) - \phi(x)\lambda_t(x, \sigma_s) \bigg] p_s^{\sigma}(\mathrm{d}x),\\
              \int_X\phi(x)\,p_0^{\sigma}(\mathrm{d}x)=&\int_X\phi(x)\zeta(\mathrm{d}x)
       \end{align*}
       for $t\in[0,T]$ and $\phi\in B_1(X)$. As a consequence of this, we can write 
       \begin{align*}
              \int_X\phi(x)\Delta_t(\mathrm{d}x)&=\int_0^t\bigg[\underbrace{\frac{\partial}{\partial s} \int_X \phi(x) \, p_s^{\sigma}(\mathrm{d}x)-\frac{\partial}{\partial s} \int_X \phi(x) \, p_s^{\sigma'}(\mathrm{d}x)}_{=\frac{\partial}{\partial s} \int_X \phi(x) \, \Delta_s(\mathrm{d}x)}\bigg]\mathrm{d}s
       \end{align*}
       Inserting an intermediate term, we find that
       \begin{align*}
              \frac{\partial}{\partial t} \int_X \phi(x) \, \Delta_t(\mathrm{d}x)
              =
              I_1(t) + I_2(t)
       \end{align*}
       with
       \begin{align*}
              I_1(t) &:= \int_X \bigg[ \int_X \phi(y) \mu_t(x,\sigma_t',\mathrm{d}y) - \phi(x)\lambda_t(x,\sigma_t')\bigg] \, \Delta_t(\mathrm{d}x), \\
              I_2(t) &:= \int_X \bigg[ \int_X \phi(y) \big(\mu_t(x,\sigma_t,\mathrm{d}y) - \mu_t(x,\sigma'_t,\mathrm{d}y)\big)\bigg] \, p^{\sigma}_t(\mathrm{d}x) \\
              &\,\quad- \int_X \phi(x)\big(\lambda_t(x,\sigma_t) - \lambda_t(x,\sigma_t')\big) \, p_t^{\sigma}(\mathrm{d}x).
       \end{align*}
       We evaluate each term separately. Define $\psi_t(x) := \int_X \phi(y) \mu_t(x,\sigma_t',\mathrm{d}y) - \phi(x)\lambda_t(x,\sigma_t')$ and $M := \sup_{t \in [0,T], x \in X, \sigma \in \mathcal{P}(X)} \lambda_t(x,\sigma) < \infty$. Note also that $\lVert \mu_t(x,\sigma_t', \cdot) \rVert_{\text{TV}} = \lambda_t(x,\sigma_t')$. Since $\phi\in B_1(X)$, it holds that
       \begin{align*}
              |\psi_t(x)|\leq \|\mu_t(x,\sigma_t',\cdot)\|_{\text{TV}}+|\lambda_t(x,\sigma_t')|=2|\lambda_t(x,\sigma_t')|\leq 2M,
       \end{align*}
       Using this result, we arrive at
       \begin{align*}
              \big| I_1(t) \big| \leq \sup_{x\in X}|\psi_t(x)|\|\Delta_t\|_{\text{TV}}
              \leq 2M  \lVert \Delta_t \rVert_{\text{TV}}.
       \end{align*}
       Invoking the triangle inequality and~(L1) yields
       \begin{align*}
              \big| I_2(t) \big| &\leq \Big(K\lVert \sigma_t-\sigma'_t \rVert_{\text{TV}} + \sup_{x\in X} |\lambda_t(x,\sigma_t) - \lambda_t(x,\sigma_t')|\Big)  \lVert p_t^{\sigma} \rVert_{\text{TV}}.
       \end{align*} 
       Recall that $p_t^{\sigma}$ is a probability measure; in particular, $\lVert p_t^{\sigma} \rVert_{\text{TV}}=1$. Further,
       \begin{align*}
              \lambda_t(x,\sigma_t) - \lambda_t(x,\sigma_t')
              =
              \mu_t(x,\sigma_t, X) - \mu_t(x,\sigma_t', X),
       \end{align*}
       which due to~(L1) implies 
       \begin{align*}
              \big| \lambda_t(x,\sigma_t) - \lambda_t(x,\sigma_t') \big| \leq \lVert \mu_t(x,\sigma_t, \cdot) - \mu_t(x,\sigma_t', \cdot) \rVert_{\text{TV}}\leq K\|\sigma_t-\sigma_t'\|_{\text{TV}}.
       \end{align*}
       Consequently, we arrive at
       \begin{align*}
              \big| I_2(t) \big| \leq 2K\lVert \sigma_t-\sigma_t' \rVert_{\text{TV}}.
       \end{align*}
       Collecting results yields
       \begin{align*}
              \bigg| \frac{\partial}{\partial t} \int_X \phi(x) \, \Delta_t(\mathrm{d}x) \bigg| \leq 2(M+K) \big(\lVert \Delta_t \rVert_{\text{TV}}+\|\sigma_t-\sigma_t'\|_{\text{TV}}\big)
       \end{align*}
       for all $\phi \in B_1(X)$, which implies
       \begin{align*}
              \|\Delta_t\|_{\text{TV}}\leq\int_0^t2(M+K)\big(\|\Delta_s\|_{\text{TV}}+\|\sigma_s-\sigma_s'\|_{\text{TV}}\big)\mathrm{d}s.
       \end{align*}
       This can be rewritten into
       \begin{align*}
              e^{-\beta t}\|\Delta_t\|_{\text{TV}}&\leq\int_0^te^{-\beta (t-s)}2(M+K)\big(e^{-\beta s}\|\Delta_s\|_{\text{TV}}+e^{-\beta s}\|\sigma_s-\sigma_s'\|_{\text{TV}}\big)\mathrm{d}s \\
              &\leq
              \frac{2(M+K)}{\beta}\bigg(\sup_{s\in[0,t]}e^{-\beta s}\|\Delta_s\|_{\text{TV}}+\sup_{s\in[0,t]}e^{-\beta s}\|\sigma_s-\sigma_s'\|_{\text{TV}}\bigg).
       \end{align*}
       Consequently,
       \begin{align*}
              \sup_{t\in[0,T]}e^{-\beta t}\|\Delta_t\|_{\text{TV}}\leq\frac{2(M+K)}{\beta}\bigg(\sup_{t\in[0,T]}e^{-\beta t}\|\Delta_t\|_{\text{TV}}+\sup_{t\in[0,T]}e^{-\beta t}\|\sigma_t-\sigma_t'\|_{\text{TV}}\bigg).
       \end{align*}
       Choosing $\beta=6(M+K) < \infty$ and rearranging yields
       \begin{align*}
              \sup_{t\in[0,T]}e^{-\beta t}\|\Delta_t\|_{\text{TV}}\leq \frac{1}{2}\sup_{t\in[0,T]}e^{-\beta t}\|\sigma_t-\sigma_t'\|_{\text{TV}},
       \end{align*}
       which is the desired contraction property (with $C = \frac{1}{2}$).
\end{proof}

We conclude with a brief comparison to the results of~\cite{Hornung2026}. Let $\lVert \nu \rVert_{\text{BL}}$ denote the bounded Lipschitz norm. In~\cite{Hornung2026}, existence of a solution to~\eqref{eq:nonlinear} is proven under the following alternative conditions:
\begin{itemize}
\item[(B1)] $\sup_{t \in [0,T], x \in X, \sigma \in \mathcal{P}(X)} \lambda_t(x,\sigma) < \infty$ 
\item[(L2)] There exists a constant $K > 0$ such that for all $t\in[0,T]$, $x,x' \in X$, and $\sigma, \sigma' \in \mathcal{P}(X)$ it holds that
\begin{align*}
\lVert \mu_t(x,\sigma,\cdot) - \mu_{t}(x',\sigma',\cdot) \rVert_{\text{TV}} \leq K \big(\mathds{1}_{\{x \neq x'\}} + \lVert \sigma - \sigma' \rVert_{\text{BL}}\big).
\end{align*}
\end{itemize}
The conditions~(L1) and~(L2) are related as follows. Since $\mu_t(x,\sigma,\mathrm{d}y)$ is a measure with $\mu_t(x,\sigma,\{x\})=0$, it holds that $\lVert \mu_t(x,\sigma, \cdot) \rVert_{\text{TV}} = \lambda_t(x,\sigma)$. Consequently, given (B1), the condition (L1) is equivalent to the following condition: There exists $C>0$ such that for all $t\in[0,T]$, $x,x' \in X$, and $\sigma, \sigma' \in \mathcal{P}(X)$ it holds that
\begin{align*}
\lVert \mu_t(x,\sigma,\cdot) - \mu_{t}(x',\sigma',\cdot) \rVert_{\text{TV}} \leq C \big(\mathds{1}_{\{x \neq x'\}} + \lVert \sigma - \sigma' \rVert_{\text{TV}}\big).
\end{align*}
Since $\lVert \nu \rVert_{\text{BL}} \leq \lVert \nu \rVert_{\text{TV}}$, we thus see that~(L2) implies~(L1), but not the converse. 

How about the remaining results in~\cite{Hornung2026}? Contrary to the existence and uniqueness of the non-linear forward integro-differential equations~\eqref{eq:nonlinear}, condition~(L2) may not be relaxed to~(L1) in regard to the convergence results stated in Theorem~4.4, Theorem~5.3, and Theorem~5.6 of~\cite{Hornung2026}. The reason for this is that convergence of the empirical distribution of random variables with non-atomic distributions breaks down when using the total variation rather than the bounded Lipschitz norm. 

\section*{Acknowledgments}

Both authors have carried out this research in association with the project frame InterAct.

\section*{Disclosure statement}

The expressed opinions are attributable solely to us and do not necessarily reflect the views of any of our past, current, and future employers.

\section*{Artificial Intelligence statement}
The Generative AI tool Gemini 3.6 Flash developed by Google has been used to assist with proof exploration and for copyediting purposes.

\bibliographystyle{amsplain}
\bibliography{references.bib}

@article{FahrenwaldtEtAl2018,
  author = {M.A. Fahrenwaldt and S. Weber and K. Weske},
  title = {{Pricing of cyber insurance contracts in a network model}},
  journal = {ASTIN Bulletin},
  year = {2018},
  volume = {48},
  issue = {3},
  pages = {1175-1218},
  doi = {10.1017/asb.2018.23}}

@article{FeinbergMandavaShiryaev2014,
  doi = {10.1016/j.jmaa.2013.09.043},
  journal = {Journal of Mathematical Analysis and Applications},
  volume = {411},
  year = {2014},
  pages = {261-270},
  title = {{On solutions of Kolmogorov's equations for nonhomogeneous jump Markov processes}},
  author = {E.A. Feinberg and M. Mandava and A.N. Shiryaev}}

@article{FeinbergMandavaShiryaev2022,
  doi = {10.1007/s10479-017-2538-8},
  journal = {Annals of Operations Research},
  volume = {317},
  issue = {2},
  year = {2022},
  pages = {587-604},
  title = {{Kolmogorov’s equations for jump Markov processes with unbounded jump rates}},
  author = {E.A. Feinberg and M. Mandava and A.N. Shiryaev}}

@article{Feller1940,
  author = {W. Feller},
  title = {{On the Integro-Differential Equations of Purely Discontinuous Markoff Processes}},
  journal = {Transactions of the American Mathematical Society},
  year = {1940},
  volume = {48},
  issue = {3},
  pages = {488-515},
  doi = {10.2307/1990095}}

@article{FengGarrdio2011,
  author = {R. Feng and J. Garrido},
  title = {{Actuarial Applications of Epidemiological Models}},
  journal = {North American Actuarial Journal},
  year = {2011},
  volume = {15},
  pages = {112-136},
  doi = {10.1080/10920277.2011.10597612}}

@inbook{FengEtAl2022,
 author = {R. Feng and J. Garrido and L. Jin and S.-H. Loke and L. Zhang},
 editor = {M.d.C. Boado-Penas and J Eisenberg and {\c{S}}. {\c{S}}ahin},
 title={{Epidemic Compartmental Models and Their Insurance Applications}},
 bookTitle={{Pandemics: Insurance and Social Protection}},
 year={2022},
 publisher={Springer International Publishing},
 pages={13-40},
 doi={10.1007/978-3-030-78334-1_2}}

@article{FrancisSteffensen2024,
  author = {L. Francis and M. Steffensen},
  title = {{Individual life insurance during epidemics}},
  journal = {Annals of Actuarial Science},
  year = {2024},
  volume = {18},
  pages = {152-175},
  doi = {10.1017/S1748499523000209}}

@article{HillairetLopez2021,
  author = {C. Hillairet and O. Lopez},
  title = {{Propagation of cyber incidents in an insurance portfolio: counting processes combined with compartmental epidemiological models}},
  journal = {Scandinavian Actuarial Journal},
  year = {2021},
  volume = {2021},
  issue = {8},
  pages = {671-694},
  doi = {10.1080/03461238.2021.1872694}}

@article{HillairetEtAl2022,
  author = {C. Hillairet and O. Lopez and L. d'Oultremont and B. Spoorenberg},
  title = {{Cyber-contagion model with network structure applied to insurance}},
  journal = {Insurance: Mathematics and Economics},
  year = {2022},
  volume = {107},
  pages = {88-101},
  doi = {10.1016/j.insmatheco.2022.08.002}}

@article{Hornung2026,
  doi = {10.1080/03461238.2026.2693555},
  journal = {Scandinavian Actuarial Journal},
  year = {2026},
  pages = {1-54},
  title = {{Mean-field approximations in insurance}},
  author = {P.C. Hornung}}

\end{document}